\documentclass[a4paper,12pt]{article}
\usepackage{amsfonts}
\usepackage[dvips]{color}
\usepackage{amsmath,amssymb,amsthm,amsfonts}
\usepackage[abbrev]{amsrefs}
\usepackage{graphicx}
\usepackage{float}   
\usepackage{caption}     
\usepackage{tikz}
\usepackage{microtype}
\usetikzlibrary{decorations.markings}
\usetikzlibrary{arrows.meta}

\theoremstyle{definition}
\newtheorem{theorem}{Theorem}[section]
\newtheorem{proposition}[theorem]{Proposition}
\newtheorem{lemma}[theorem]{Lemma}
\newtheorem{corollary}[theorem]{Corollary}
\newtheorem{remark}[theorem]{Remark}
\newtheorem{definition}[theorem]{Definition}

\newcommand{\C}{{\mathbb C}}

\newcommand{\PP}{{\mathbb P}}

\makeatletter

\@addtoreset{equation}{section}

\newcounter{def}[section]
\renewcommand{\thedef}{\stepcounter{def}\thesection.\@arabic\c@def }

\begin{document}
\setlength{\baselineskip}{24pt}
\begin{center}
\textbf{\LARGE{Kovalevskaya exponents of the Riccati hierarchy}}
\end{center}

\setlength{\baselineskip}{14pt}

\begin{center}
Advanced Institute for Materials Research, Tohoku University,\\
Sendai, 980-8577, Japan
\\[0.5cm]
\large{Changyu Zhou\footnote{E-mail address: zhou.changyu.r8@dc.tohoku.ac.jp}}
\end{center}
\begin{center}
\today
\end{center}

\begin{abstract}
We carry out a Kovalevskaya analysis of the Riccati
hierarchy. We determine all indicial loci and
Kovalevskaya exponents and identify a rigid recursive structure governing how
free parameters enter Laurent solutions. We further identify a
nontrivial quasi--homogeneous vector field commuting with the hierarchy and
use it to obtain an explicit parametrization of all solutions in terms of a
single polynomial. Notably, in the two-dimensional case, the same general solution is recovered by the blow-up resolution. Within this parametrization, collisions of poles correspond to degeneration limits of the principal Laurent family, through which lower indicial loci appear. Finally, negative Kovalevskaya exponents are interpreted analytically through annular Laurent
expansions, which describe how different collections of poles dominate in
different complex regions.
\end{abstract}

\section{Introduction}
Integrable systems play a central role in the theory of nonlinear differential equations,
revealing deep connections between algebraic geometry, complex analysis, and mathematical
physics \cites{Adl,Goriely}. A fundamental analytic tool in this theory is the notion of
\emph{Kovalevskaya exponents}, introduced in Kovalevskaya’s work on rigid body motion
\cite{Kovalevski}. These exponents arise from the linearization of formal Laurent expansions
near movable singularities and determine how many free parameters enter local solutions.
Their integrality properties form the backbone of the Painlev\'e test for integrability
\cites{Yoshida,HuYan}.

In Painlevé theory, a differential equation is said to have the Painlevé property if its general solution can be made single-valued in the complex plane\cite{Painleve}. To put this in a simplified way, the Painlevé property implies that the only movable singularities of the solution are poles, meaning that the solution has no essential singularities or branch points. The Painlevé test provides a necessary condition for this property: for a system to be integrable, its Kovalevskaya exponents must be integer values, which corresponds to the number of free parameters in the Laurent series solution\cites{Goriely, Chi1}.

In recent years, the theory of \emph{quasi--homogeneous vector fields} has provided a
systematic framework for organizing these ideas. Their intrinsic scaling structure allows
one to define indicial loci, Kovalevskaya matrices, and resonance conditions in a canonical
way, leading to a refined Painlev\'e theory in which formal Laurent series are shown to be
convergent and parametrized by finitely many free parameters \cite{Chi1}. 

If one assumes the existence of a hierarchy of integrable equations, this is equivalent to the existence of a family of commuting flows. In \cite{ZhouChiba2025}, commuting quasi–homogeneous vector fields were studied in a general framework, revealing strong constraints on the associated Kovalevskaya exponents and the structure of free parameters. Importantly, \cite{ZhouChiba2025} also explores the degeneration mechanism that connects the principal indicial loci to lower indicial loci, providing a deeper understanding of how these transitions occur within the hierarchy. Motivated by this theory and its close connection with Painlevé equations, we focus on the Riccati hierarchy, a fundamental and ubiquitous family of integrable quasi–homogeneous systems.

The Riccati hierarchy arises naturally by taking successive covariant derivatives of a
projective vector field and includes, as special cases, many classical equations in the
Chazy--Bureau classification. While several individual members of the hierarchy are well
known, their global Kovalevskaya--Painlev\'e structure has not previously been
analyzed in a systematic way. In particular, no complete description of all indicial
loci, Kovalevskaya exponents, and their analytic meaning has been available for the
hierarchy as a whole. This lack of a unified resonance theory obscures the global
organization of Laurent families and the role of negative exponents in the hierarchy.

The purpose of this paper is to fill this gap. Our main results may be summarized as
follows:

\begin{itemize}
\item \textbf{Indicial loci and Kovalevskaya exponents.}
We carry out a complete Kovalevskaya analysis of the hierarchy by rewriting it as a quasi--homogeneous first--order system. This allows us to determine explicitly all indicial loci and Kovalevskaya matrices, leading to a full classification of principal and lower branches. The resulting spectrum of Kovalevskaya exponents exhibits a rigid recursive structure. In particular, negative resonances arise in a highly organized manner, governing how free parameters enter Laurent families along different branches.

\item \textbf{Commuting symmetry.}
 We identify a nontrivial quasi--homogeneous
polynomial vector field commuting with the Riccati flow and use it to perform a
symmetry reduction, leading to a triangular normal form and an explicit
representation of the general solution. In the two--dimensional case, we show that this construction coincides with the blow--up resolution, thereby linking algebraic symmetry reduction with geometric desingularization.

\item \textbf{Degeneration mechanism.}
We show that the tower of indicial loci arises through an analytic degeneration mechanism: the collisions of poles of the solution. In the symmetry--reduced representation, rewriting the solutions in partial fraction form introduces pole coordinates corresponding to the pole locations. 
The commuting symmetry induces a dynamical system on these poles, whose singular locus is precisely the discriminant where poles collide. 
This degeneration produces lower Laurent branches, reflecting a loss of free parameters in the Laurent expansion.

\item \textbf{Annular Laurent expansions.}
We investigate how negative Kovalevskaya exponents emerge from Laurent expansions in annular regions of the complex plane. Specifically, we show that the local expansion near a pole within an annular domain depends on whether other poles lie inside or outside the region of expansion, revealing the structure behind the negative exponents.
\end{itemize}

Taken together, these results provide an analytic and structural understanding of the Riccati hierarchy: commuting symmetries generate explicit solutions, their Laurent expansions are organized by indicial loci, the degeneration of these families is controlled by the discriminant geometry of pole collisions, and negative Kovalevskaya exponents are interpreted through annular Laurent expansions.

\section{Quasi-homogeneity of the Riccati hierarchy}

In this section, we introduce the theoretical framework for the Riccati hierarchy, emphasizing its recursive structure and the central concept of quasi-homogeneity. Quasi-homogeneity plays a crucial role in the integrability of differential systems, particularly in systems with the Painlevé property, by governing the scaling behavior of their solutions.

We begin by recalling the operator formulation of the Riccati hierarchy. Throughout this chapter, $u=u(z)$ denotes a scalar
function and $c\in\mathbb C^\ast$ is a fixed nonzero constant.

\begin{definition}[\cite{Guha2006}]
Let \(L\) be the first–order differential operator
\[
L=\frac{d}{dz}+c\,u(z).
\]
The \(m\)-th equation of the Riccati hierarchy is defined by
\[
L^{m}u(z)+\sum_{j=1}^{m}\alpha_j(z)\,L^{\,j-1}u(z)+\alpha_0(z)=0,
\]
where $m\in\mathbb N$ and the coefficient functions $\alpha_j(z)$ are given.
\end{definition}

\begin{remark}
The iterates of $L$ are defined recursively by, for $i\ge 1$,
\[
L^{i}u=(\partial_z+c\,u)\,L^{i-1}u,
\]
so $L^{m}u$ is a differential polynomial in $u$ and its derivatives of order at
most $m$. 
\end{remark}

%%%%%%%%%%%%%%%
Next, we rewrite the Riccati hierarchy as a first–order system and
exhibit its quasi--homogeneous structure.

By defining
\[
x_i = L^{i-1}u ,\qquad 1\le i\le m,
\]
the scalar Riccati hierarchy can be converted into a system of first–order
equations. A direct computation gives
\begin{equation}
\label{eq:RC-system}
\begin{aligned}
x_i' &= x_{i+1} - c x_1 x_i, \quad &&1 \le i \le m-1, \\
x_m' &= - \sum_{j=0}^{m} \alpha_j(z) x_j - c x_1 x_m, \quad && (x_0 \equiv 1),
\end{aligned}
\end{equation}
which reduces, when $m=1$, to the standard Riccati equation
\[
x_1' = -\alpha_1(z) x_1 - \alpha_0(z) - c x_1^2.
\]

\subsection{Quasi-homogeneous structure}

The structure of the Riccati hierarchy \eqref{eq:RC-system} can be better understood by considering the weights of its terms. We assign weights to the variables \(x_i\) by
\[
\mathrm{wt}(x_i)=p_i:=i,\qquad 1\le i\le m,
\]
and require the associated vector field
$f=(f_1,\dots,f_m)$ to be quasi--homogeneous of degree $+1$ with respect to
this grading, meaning that it satisfies the scaling relation:
\[
f_i(\lambda^{p_1} x_1, \dots, \lambda^{p_m} x_m,\lambda^rz) = \lambda^{p_i+1} f_i(x_1, \dots, x_m,z), \quad i = 1, \dots, m, \quad \lambda \in \mathbb{C}^*.
\]

For $1\le i\le m-1$, the equations
\[
x_i' = x_{i+1} - c x_1 x_i
\]
are automatically compatible with this choice of weights, since
\[
\mathrm{wt}(x_{i+1})=\mathrm{wt}(x_1x_i)=i+1.
\]

For the last equation in \eqref{eq:RC-system}, quasi--homogeneity imposes
conditions on the coefficient functions $\alpha_j$.
Since \(
\mathrm{wt}(x_m')=m+1,
\) each term $\alpha_j x_j$ must have the same total weight.
This yields the weight assignment
\[
\mathrm{wt}(\alpha_j)=m+1-j,\qquad 0\le j\le m.
\]
Thus, the quasi--homogeneous structure of the Riccati hierarchy uniquely
determines the weights of the coefficient functions.

%%%%%%%%%%%%%%%%
\subsection{Weight classification}

We now introduce a weight for the variable \(z\),
\[
\mathrm{wt}(z)=r\in\{1,\dots,m+1\},
\]
and assume that each $\alpha_j(z)$ is a polynomial in $z$.
Compatibility with the assigned weight
$\mathrm{wt}(\alpha_j)=m+1-j$ then implies that $\alpha_j(z)$ must take the form
\[
\alpha_j(z)=
\begin{cases}
A_j z^{(m-j+1)/r}, & \text{if } \dfrac{m-j+1}{r}\in\mathbb Z,\\[6pt]
0, & \text{otherwise},
\end{cases}
\qquad 0\le j\le m.
\]

\subsubsection{Examples for $m=1$}

For $m=1$, the hierarchy reduces to the standard Riccati equation
\[
x_1'=-\alpha_1(z)x_1-\alpha_0(z)-cx_1^2.
\]
The above weight classification leads to exactly two quasi--homogeneous cases,
namely
\[
x_1'=-z-cx_1^2 \qquad \text{(Airy type)}
\]
and
\[
x_1'=-zx_1-z^2-cx_1^2 \qquad \text{(Hermite type)}.
\]
The Airy case was studied in \cite{Chi}, where a compactified Riccati equation of
Airy type on a weighted projective space is constructed.

\subsubsection{The case $m=2$}

For $m=2$ the scalar Riccati hierarchy yields the second--order equation
\[
u'' + (\alpha_2(z)+3cu)\,u' + c^2u^3 + c\alpha_2(z)u^2 + \alpha_1(z)u + \alpha_0(z)=0.
\]
In the simplest case $\alpha_j(z)=0$ and $c=1$ this reduces to
\[
u'' + 3u\,u' + u^3 = 0,
\]
and differentiating gives the Chazy IV equation
\[
u''' + 3u u'' + 3(u')^2 + 3u^2u' = 0.
\]

\subsubsection{Higher-order examples and Chazy--Bureau equations}

For higher values of $m$, the same weight mechanism produces nontrivial
higher--order equations of Painlev\'e type. For instance, when $m=3$ the scalar
Riccati hierarchy takes the form
\[
\begin{aligned}
u''' &+ (\alpha_3(z) + 4c u) u'' + 3c (u')^2 \\
&+ \left(6c^2 u^2 + 3c \alpha_3(z) u + \alpha_2(z)\right) u' \\
&+ c^3 u^4 + c^2 \alpha_3(z) u^3 + c \alpha_2(z) u^2 + \alpha_1(z) u + \alpha_0(z) = 0.
\end{aligned}
\]

In the simplest case $\alpha_j(z)=0$ and $c=1$, this reduces, after an affine
change of variables, to the Chazy XII equation
\[
u''' + 4u u'' + 3(u')^2 + 6u^2u' + u^4 = 0.
\]

For $m=4$, the hierarchy yields the Fuchs--type equation
\[
u^{(4)} + 5u u''' + 10u'u'' + 10u^2u'' + 15u(u')^2 + 10u^3u' + u^5 = 0,
\]
which coincides with the Fuchs equation F--XVI in the Bureau classification.

These higher--order members of the Riccati hierarchy and their relations to
Painlev\'e--type equations have been studied in detail in \cite{Guha2006}, where
they arise from the stabilizer set of a Virasoro coadjoint orbit.
%%%%%%%%%%%%%%%
%%%%%%%%%%%%%%%%%%%%%%%%%%%%%%%
\section{Indicial loci and Kovalevskaya exponents}
In this section, we determine all indicial loci and Kovalevskaya exponents of the autonomous Riccati hierarchy, obtaining a complete classification of principal and lower loci.

We begin by considering the autonomous part of the \( m \)-th Riccati hierarchy. We assume \(\alpha_0= 0\). Then  \eqref{eq:RC-system} written as an \( m \)-dimensional first-order system:
\[
\frac{dx_i}{dz} = f_i^A(x), \quad i = 1, \dots, m,
\]
where the functions \( f_i^A(x) \) are given by
\begin{equation}
\label{eq:RC-aut}
\begin{aligned}
  f_i^A(x) &= x_{i+1} - c x_1 x_i, \quad && 1 \le i \le m-1, \\
  f_m^A(x) &= - c x_1 x_m,
\end{aligned}
\end{equation}
and \( c \in \mathbb{C}^* \) is a constant.

We assume that any non-autonomous terms in the full system \eqref{eq:RC-system} are subdominant near a
movable singularity; hence the indicial equations are determined by the
autonomous system \eqref{eq:RC-aut}.

Next, we examine the formal Laurent series solutions of the truncated system of equations for small \( z - z_0 \). By quasihomogenity of the system, the solutions are of the form
\[
x_i(z) = \xi_i (z - z_0)^{-p_i} + \cdots,
\]
where \( p_i = i \) for \( i = 1, \dots, m \), and the leading coefficients \( \xi = (\xi_1, \dots, \xi_m) \) are called indicial locus and satisfy the indicial equations
\begin{equation}
\label{eq:indicial-general}
-p_i \xi_i = f_i^A(\xi), \quad i = 1, \dots, m.
\end{equation}

\subsection{Indicial loci}
\begin{lemma}
\label{lem:indicial-loci}
For each integer \( n \in \{1, \dots, m\} \), we have
\[
\xi^{(n)} = \left(\xi^{(n)}_1, \dots, \xi^{(n)}_m \right), \quad
\xi^{(n)}_i =
\begin{cases}
\dfrac{(n)_i}{c}, & 1 \le i \le n, \\
0, & n < i \le m,
\end{cases}
\]
where \( (n)_i := n(n-1)\cdots(n-i+1) \) is the falling factorial. Then each \( \xi^{(n)} \) satisfies the indicial equations \eqref{eq:indicial-general}. Moreover, these vectors \( \xi^{(1)}, \dots, \xi^{(m)} \) are the only nonzero solutions to the indicial equations.
\end{lemma}

\begin{proof}
Substituting \( x_i(z) = \xi_i (z - z_0)^{-i} \) into \eqref{eq:indicial-general} and comparing the leading terms, we obtain the following recursive relations:
\[
-i \xi_i = \xi_{i+1} - c \xi_1 \xi_i, \quad 1 \le i \le m - 1, \quad
-m \xi_m = -c \xi_1 \xi_m.
\]
From the second equation, we have
\[
(c \xi_1 - m) \xi_m = 0.
\]
If \( \xi_m \ne 0 \), then \( \xi_1 = \frac{m}{c} \). Using the recursive structure of the system, we find that \( \xi_i = \frac{(m)_i}{c} \) for \( 1 \le i \le m \).

On the other hand, if \( \xi_m = 0 \), the condition \( (c \xi_1 - (m-1)) \xi_{m-1} = 0 \) must hold. Thus, either \( \xi_1 = \frac{m-1}{c} \) or \( \xi_{m-1} = 0 \). If \( \xi_1 = \frac{m-1}{c} \), we obtain \( \xi_i = \frac{(m-1)_i}{c} \) for \( 1 \le i \le m-1 \). Continuing this process, we obtain \( \xi_i = \frac{(n)_i}{c} \) for some integer \( n \in \{1, \dots, m\} \).

Thus, the only nonzero solutions to the indicial equations are the vectors \( \xi^{(1)}, \dots, \xi^{(m)} \), which correspond to the falling factorials \( (n)_i/c \) for each \( n \).
\end{proof}

%%%%%%%%%%%%%%%%%%%%
\subsection{Kovalevskaya matrix}

For a given indicial locus $\xi\in\C^m\setminus\{0\}$, the
\emph{Kovalevskaya matrix} is defined as
\begin{equation}
\label{eq:K-def}
K(\xi) := \bigl(K_{ij}(\xi)\bigr)_{1\le i,j\le m},\qquad
K_{ij}(\xi) = \frac{\partial f_i^A}{\partial x_j}(\xi) + p_i\delta_{ij},
\end{equation}
where $p_i=i$ and $\delta_{ij}$ is the Kronecker delta. Its eigenvalues are
called the Kovalevskaya exponents associated with $\xi$.

\begin{theorem}
The eigenvalues of \( K(\xi^{(n)}) \) are given by:
\[
\lambda_i = -n + i, \quad \text{for} \quad i = 0, 1, 2, \dots, n-1,
\]
and
\[
\lambda_i = - n + i +1, \quad \text{for} \quad i = n, n+1, \dots, m-1.
\]
\end{theorem}

\begin{proof}
Let \(K(\xi^{(n)})\) be the Kovalevskaya matrix at the indicial locus \(\xi^{(n)}\).
Then \(K(\xi^{(n)})-\lambda I\) has the form
\[
\begin{pmatrix}
1-2n-\lambda & 1 & 0 & \cdots & 0\\
-(n)_2 & 2-n-\lambda & 1 & \cdots & 0\\
-(n)_3 & 0 & 3-n-\lambda & \ddots & \vdots\\
\vdots & \vdots & \ddots & \ddots & 1\\
-(n)_m & 0 & \cdots & 0 & m-n-\lambda
\end{pmatrix},
\]
where \((n)_i\) denotes the falling factorial. For \(i>n\), we have \((n)_i=0\). Hence \(K(\xi^{(n)})-\lambda I\) decomposes into a block diagonal matrix
\[
K(\xi^{(n)})-\lambda I=
\begin{pmatrix}
A_n & 0\\
0 & C_{m-n}
\end{pmatrix},
\]
where \(A_n\) is an \(n\times n\) matrix and \(C_{m-n}\) is upper triangular.
Therefore,
\[
\det(K(\xi^{(n)})-\lambda I)=\det(A_n)\det(C_{m-n}).
\]
Since \(C_{m-n}\) is upper triangular with diagonal entries
\(
m-n-\lambda,\dots,m-1-\lambda,
\)
we obtain
\[
\det(C_{m-n})=\prod_{i=n+1}^{m}(i-n-\lambda).
\]
The matrix \(A_n\) has the form
\[
A_n=
\begin{pmatrix}
1-2n-\lambda & 1 & 0 & \cdots & 0\\
-(n)_2 & 2-n-\lambda & 1 & \cdots & 0\\
\vdots & \vdots & \ddots & \ddots & 1\\
-(n)_n & 0 & \cdots & 0 & -\lambda
\end{pmatrix}.
\]
%%%%%%%%%%%%%%%%
It follows that
\begin{align*}
\det(A_n)
&=(1-2n-\lambda)\det\!\bigl(A_n^{(1,1)}\bigr)+(n)_2\det\!\bigl(A_n^{(1,2)}\bigr)\\
&=(1-2n-\lambda)\prod_{i=2}^{n}(i-n-\lambda)
\;+\;(n)_2\prod_{i=3}^{n}(i-n-\lambda)\\
&\qquad+\sum_{j=3}^{n-1}(-1)^{j}(n)_j\prod_{i=j+1}^{n}(i-n-\lambda)
\;+\;(-1)^n(n)_n\\
&=(1-2n-\lambda)\prod_{i=2}^{n}(i-n-\lambda)
\;+\;(n)_2\prod_{i=2}^{n-1}(i-n-\lambda)\\
&=\Bigl((1-2n-\lambda)(-\lambda)+n(n-1)\Bigr)\prod_{i=2}^{n-1}(i-n-\lambda)\\
&=\prod_{i=0}^{n-1}(i-n-\lambda).
\end{align*}
%%%%%%%%%%%%
We obtain
\[
\det(A_n)=\prod_{i=0}^{n-1}(i-n-\lambda).
\]
%%%%%%%%%%%%%%%%
Combining these expressions yields
\[
\det(K(\xi^{(n)})-\lambda I)
=\prod_{i=0}^{n-1}(i-n-\lambda)\prod_{i=n+1}^{m}(i-n-\lambda).
\]
Thus the eigenvalues of \(K(\xi^{(n)})\) are
\[
\lambda=-n,-n+1,\dots,-1,1,\dots,m-n,
\]
as claimed.
\end{proof}

The preceding theorem gives the full spectrum of the Kovalevskaya matrix at each
balance $\xi^{(n)}$. We now interpret this spectrum in terms of the number of
resonant parameters in the associated Laurent expansions, which leads to a
natural distinction between the principal balance and lower balances.

\subsection{Principal and lower indicial loci}
Let \(\kappa(\xi^{i})\) denote the Kovalevskaya exponents associated with the indicial locus \(\xi^{i}\) for \(i\in\{1,\dots,m\}\). We can write the Kovalevskaya exponents of Riccati hierarchy as following:

\begin{align*}
    \text{Principal} \quad &\kappa(\xi^{1}) = \{-1, 1, \dots, m-1\}, \\
    \text{Lower} \quad &\kappa(\xi^{2}) = \{-2, -1, 1, \dots, m-2\}, \\
    &\vdots \\
    \text{Lower} \quad &\kappa(\xi^{m-1}) = \{-m+1, \dots, 1, -1, 1\}, \\
    \text{Lower} \quad &\kappa(\xi^{m}) = \{-m, \dots, 1, -1\}.
\end{align*}

\begin{definition}
An indicial locus $\xi$ is called principal if the Kovalevskaya matrix
has exactly $m-1$ nonnegative integer Kovalevskaya exponents. If the number of nonnegative integer Kovalevskaya exponents is smaller than $m-1$,
the locus is called a lower indicial locus.
\end{definition}

With this definition in place, we can identify the principal and lower loci for the
Riccati hierarchy directly from the explicit spectra above. In the context of Kovalevskaya exponents of Riccati hierarchy, the indicial locus $\xi^{(1)}$ is \emph{principal}, since its K-exponents
are exactly $\{-1,1,\dots,m-1\}$, i.e.\ one distinguished eigenvalue $-1$ and
$m-1$ strictly positive integers. The loci $\xi^{(n)}$ with $n\ge2$ are
called \emph{lower} indicial loci, since they do not carry the full set of nonnegative integer
Kovalevskaya exponents.

\begin{remark}
\label{rem:riccati_negK_leading}
 On each Riccati-hierarchy Laurent branch determined by $\xi=\xi^{(n)}$, the branch parameter $\xi_1$
determines the smallest negative Kovalevskaya exponent $-n$ via $-n=-c\xi_1$.
\end{remark}
Equivalently, this smallest negative exponent can be seen directly at the
level of the Kovalevskaya matrix: the balance vector itself furnishes an eigenvector.

\begin{corollary}\label{cor:evec-minus-n}
For the Riccati hierarchy at the indicial locus $\xi^{(n)}$, the vector $\xi^{(n)}$ itself
is an eigenvector of $K(\xi^{(n)})$ with eigenvalue $-n$, i.e.
\[
K(\xi^{(n)})\,\xi^{(n)}=-n\,\xi^{(n)}.
\]
\end{corollary}

\begin{proof}
In the explicit matrix form of $K(\xi^{(n)})-\lambda I$, set $\lambda=-n$.
Since $(n)_i=0$ for $i>n$, the lower-right block does not interact with the first $n$ components.
A direct substitution (using the same falling-factorial recursion as in the indicial equations)
shows $(K(\xi^{(n)})+nI)\,\xi^{(n)}=0$.
\end{proof}

Together with Lemma~\ref{lem:indicial-loci} and the eigenvalue formula above, this
completes the classification of indicial loci and Kovalevskaya exponents for the
autonomous Riccati hierarchy.
%%%%%%%%%%%%%%%%

%%%%%%%%%%%%%%%%%%%%%%%%%%
\section{Commuting symmetry and degeneration}
In this section, we show that the Riccati hierarchy admits a nontrivial
quasi–homogeneous polynomial vector field commuting with the flow. This commuting symmetry generates invariants and induces a symmetry reduction to a triangular normal form, which provides an explicit parametrization of the general solution. Furthermore, we demonstrate how this symmetry governs the degeneration of the space of Laurent solutions through pole collision. In the two-dimensional case, we show that this degeneration process is geometrically realized by a blow-up resolution, revealing that the symmetry reduction aligns with the natural regularization of the foliation at infinity.

We recall the autonomous Riccati hierarchy on $\C^m$,
\begin{equation}
f_i(x)=x_{i+1}- x_1x_i,\qquad 1\le i\le m-1,
\qquad 
f_m(x)=-x_1x_m,
\label{riccati-autc1}
\end{equation}
where we assume $c=1$ for simplicity. Throughout this section we assume the quasi-homogeneous weight
\[
\mathrm{wt}(x_i)=i,\qquad i=1,\dots,m.
\]
From now on, we consider a partial differential equations system on $\mathbb{C}^m$
\[
\begin{cases}
\frac{\partial x_i}{\partial \boldsymbol{z_1}} = f_i(x), \quad i=1,\dots, m,\\
\frac{\partial x_i}{\partial \boldsymbol{z_2}} = g_i(x), \quad i=1,\dots, m,
\end{cases}
\]
where $G=(g_1,\dots,g_m)$, and $G$ is quasi-homogeneous of degree $\gamma\in\mathbb{N}$ with respect to\ $\mathrm{wt}(x_i)=i$, i.e.
\[
g_i(\lambda x_1,\lambda^2 x_2,\dots,\lambda^m x_m)=\lambda^{i+\gamma}g_i(x),
\qquad \lambda\in\mathbb{C}.
\]
Furthermore, we assume $[F,G]=0$, i.e.
\begin{equation}
    [F,G]_i=\sum_{j=1}^m\Bigl(f_j\,\partial_{x_j}g_i-g_j\,\partial_{x_j}f_i\Bigr)=0, \qquad i=1,\dots,m.
    \label{eq:commute}
\end{equation}

\subsection{A commuting symmetry}
For the symmetry reduction in the next subsection, it suffices to exhibit a
quasi-homogeneous polynomial vector field commuting with the Riccati hierarchy.
The following example will be used throughout.

\begin{proposition}\label{prop:commuting-G}
Let \(F=(f_1,\dots,f_m)\) be the autonomous Riccati hierarchy and define
\(G=(g_1,\dots,g_m)\) by
\[
g_i=x_i x_m\ (1\le i\le m-1),\qquad g_m=x_m^2.
\]
Then \(G\) commutes with \(F\), i.e.\ \([F,G]=0\). Moreover, \(G\) is nontrivial:
it is not a scalar multiple of \(F\), and it is quasi-homogeneous of degree
\(\gamma=m\).
\end{proposition}

\begin{proof}
Using \eqref{eq:commute}, we obtain the followings.
For \(i<m\), \(g_i=x_i x_m\) implies \(\partial_{x_i}g_i=x_m\), \(\partial_{x_m}g_i=x_i\), hence
\[
\sum_j f_j\partial_{x_j}g_i=f_i x_m+f_m x_i.
\]
Also \(f_i=x_{i+1}-x_1x_i\) implies \(\partial_{x_{i+1}}f_i=1\), \(\partial_{x_1}f_i=-x_i\), \(\partial_{x_i}f_i=-x_1\), hence
\[
\sum_j g_j\partial_{x_j}f_i=g_{i+1}-x_i g_1-x_1 g_i.
\]
Substituting \(f_i=x_{i+1}-x_1x_i\), \(f_m=-x_1x_m\), and \(g_i=x_ix_m\) gives
\[
f_i x_m+f_m x_i=g_{i+1}-x_i g_1-x_1 g_i,
\]
so \([F,G]_i=0\) for \(i<m\).

For \(i=m\), \(g_m=x_m^2\) and \(f_m=-x_1x_m\) give
\[
[F,G]_m=2x_m f_m-\bigl((-x_m)g_1+(-x_1)g_m\bigr),
\]
which vanishes after substituting \(g_1=x_1x_m\), \(g_m=x_m^2\). Thus \([F,G]=0\).
\end{proof}

\subsection{Symmetry reduction}

The commuting vector field $G$ admits a family of rational invariants along its
flow. Indeed, since
\[
g_i = x_i x_m,\qquad g_m = x_m^2,
\]
it follows that the ratios
\[
y_i := \frac{x_i}{x_m}, \qquad i=1,\dots,m-1,
\]
are invariant under the $G$--flow. This observation allows the Riccati
hierarchy to be reduced to a triangular system.

\begin{corollary}
On the open set $\{x_m\neq 0\}$, the ratios $y_i:=\frac{x_i}{x_m}$ are invariants
along the $G$-flow:
\[
\frac{\partial}{\partial z_2}\!\left(\frac{x_i}{x_m}\right)=0,
\qquad i=1,\dots,m-1.
\]
\end{corollary}

\begin{proof}
For $y_i=x_i/x_m$,
\[
\frac{\partial y_i}{\partial z_2}
=
\frac{\partial}{\partial z_2}\!\left(\frac{x_i}{x_m}\right)
=\frac{g_i x_m-x_i g_m}{x_m^2}
=\frac{(x_i x_m)x_m-x_i(x_m^2)}{x_m^2}=0.
\]
\end{proof}

\begin{corollary}\label{cor:G-reduction-solution}
On the open set \(\{x_m\ne 0\}\), define
\begin{equation}\label{eq:yu-def}
y_i:=\frac{x_i}{x_m}\quad (i=1,\dots,m-1),\qquad
u:=\frac{1}{x_m}.
\end{equation}
Then the \(z_1\)-flow \(\partial_{z_1}x=F(x)\) is transformed into the triangular system
\begin{equation}\label{eq:triangular-yu}
\partial_{z_1} y_i = y_{i+1}\quad (1\le i\le m-2),\qquad
\partial_{z_1} y_{m-1} = 1,\qquad
\partial_{z_1} u = y_1.
\end{equation}
Consequently,
\begin{equation}\label{eq:u-explicit-poly}
u(z_1,z_2)=\frac{z_1^m}{m!}+a_{m-1}(z_2)\frac{z_1^{m-1}}{(m-1)!}
+\cdots+a_1(z_2) z_1+a_0(z_2),
\end{equation}
and
\begin{equation}\label{eq:yi-explicit-poly}
y_i(z_1,z_2)=\partial_{z_1}^i u(z_1,z_2)\qquad (i=1,\dots,m-1).
\end{equation}
Hence the general solution of \eqref{riccati-autc1} (for the \(z_1\)-flow, with \(z_2\) as a parameter) is
\begin{equation}\label{eq:x-explicit-from-u}
x_m(z_1,z_2)=\frac{1}{u(z_1,z_2)},\qquad
x_i(z_1,z_2)=\frac{y_i(z_1,z_2)}{u(z_1,z_2)}
=\frac{\partial_{z_1}^i u(z_1,z_2)}{u(z_1,z_2)}\quad (i=1,\dots,m-1).
\end{equation}
\end{corollary}

\begin{proof}
Define \(y_i\) and \(u\) by \eqref{eq:yu-def}. Using the \(z_1\)-equations
\(\partial_{z_1}x_i=f_i(x)\), for \(1\le i\le m-2\) we compute
\[
\partial_{z_1}y_i
=\frac{(\partial_{z_1}x_i)x_m-x_i(\partial_{z_1}x_m)}{x_m^2}
=\frac{x_{i+1}}{x_m}
=y_{i+1}.
\]
Similarly,
\[
\partial_{z_1}y_{m-1}
=\frac{(\partial_{z_1}x_{m-1})x_m-x_{m-1}(\partial_{z_1}x_m)}{x_m^2}
=\frac{x_m^2}{x_m^2}=1,
\]
and
\[
\partial_{z_1}u
=\partial_{z_1}\!\left(\frac{1}{x_m}\right)
=-\frac{\partial_{z_1}x_m}{x_m^2}
=-\frac{-x_1x_m}{x_m^2}
=\frac{x_1}{x_m}
=y_1.
\]
This proves \eqref{eq:triangular-yu}.

Integrating \eqref{eq:triangular-yu} in \(z_1\) (with \(z_2\) fixed) gives
\(y_{m-1}(z_1,z_2)=z_1+a_{m-1}(z_2)\), and inductively each \(y_i(\cdot,z_2)\) is a
polynomial in \(z_1\) of degree \(m-i\). Since \(\partial_{z_1}u=y_1\), it follows
that \(u(\cdot,z_2)\) is a degree-\(m\) polynomial in \(z_1\) of the form
\eqref{eq:u-explicit-poly}, and that \(y_i=\partial_{z_1}^i u\).
Substituting these relations into \eqref{eq:yu-def} yields \eqref{eq:x-explicit-from-u}.
\end{proof}

\subsection{Resolution by blow-up}
The blow-up method, commonly used in Painlev\'e analysis to resolve singularities
\cites{Okamoto, Sakai}, introduces a sequence of birational transformations that
compactify the phase space and remove indeterminacies at infinity. In the context of the 2-dimensional Riccati hierarchy, we apply the blow-up method. This procedure provides a regular foliation on a smooth model and will be shown to yield the same general solution \eqref{eq:x1-xm-poles} as derived through symmetry reduction. 

\subsubsection{Compactification and a base point at infinity}
Consider the $2$-dimensional autonomous system
\begin{equation}\label{eq:2d-system}
x_1'=x_2-x_1^2,\qquad x_2'=-x_1x_2,
\end{equation}
where $x_i'=\partial x_i/\partial z_1$.

On $\PP^1\times\PP^1$ introduce the chart at infinity
\[
(X,Y):=\Bigl(\frac{1}{x_1},\frac{1}{x_2}\Bigr),
\]
so that the induced vector field is
\[
X' = 1-\frac{X^2}{Y},\qquad Y'=\frac{Y}{X}.
\]
The point $(X,Y)=(0,0)$ is a base point, since the right-hand sides are indeterminate there.

\subsubsection{Blow-up}
Blow up $(X,Y)=(0,0)$. In the chart $X=u,\ Y=uv$ one obtains
\begin{equation}
u'=\frac{v-u}{v},\qquad v'=1.
\label{eq:uv-system}
\end{equation}
Equation \eqref{eq:uv-system} is  integrable:
\[
v(z)=z_1,\qquad
u(z)=\frac{z_1^2+2C}{2z_1},
\]
where $C$ is an arbitrary constant.

After transforming back to $(x_1,x_2)$, we recover the same general solutions as \eqref{eq:x1-xm-poles}.

\subsection{Degeneration by pole collision}\label{sec:riccati-degeneration}
\medskip
Having obtained an explicit parametrization of the general solution via symmetry reduction, we next analyze how this family behaves under the commuting symmetry.  
Letting $z_2$ vary makes the parameters $t_j(z_2)$ evolve, and the resulting parameter dynamics becomes singular precisely when poles collide.  
This provides a concrete analytic mechanism for degeneration from the principal Laurent family to lower families. 

Fix $z_2$ and regard $u(\cdot,z_2)$ as a degree-$m$ polynomial in $z_1$.
For generic parameters it factorizes as
\[
u(z_1,z_2)=\frac{1}{m!}\prod_{\ell=1}^m (z_1-t_\ell),
\]
with distinct roots $t_\ell=t_\ell(z_2)$. Recalling from
\eqref{eq:x-explicit-from-u} that $x_m=1/u$ and $x_1=(\partial_{z_1} u)/u$,
we obtain the following partial--fraction forms.

\begin{equation}\label{eq:x1-xm-poles}
x_1(z_1)=\sum_{k=1}^m\frac{1}{z_1-t_k},
\qquad
x_m(z_1)=\frac1{\prod_{\ell=1}^m(z_1-t_\ell)}.
\end{equation}
The commuting vector field $g_1$ acts on this family by
\[
g_1(x)=x_1x_m.
\]
\begin{proposition}
In the pole coordinates $(t_1,\dots,t_m)$, the $g_1$--flow induces the ODE system
\[
\frac{d t_j}{d z_2}=\frac{1}{\prod_{k\ne j}(t_j-t_k)},\qquad j=1,\dots,m.
\]
This system is singular exactly on the discriminant
\[
\Delta=\Bigl\{(t_1,\dots,t_m):\ \prod_{i<j}(t_i-t_j)^2=0\Bigr\}.
\]
\end{proposition}

\begin{proof}
Recall
\[
x_1(z_1,z_2)=\sum_{k=1}^m\frac{1}{z_1-t_k(z_2)},\qquad
x_m(z_1,z_2)=\frac{1}{\prod_{k=1}^m (z_1-t_k(z_2))},
\]
and $g_1=\partial/\partial z_2$ (with $z_1$ fixed) satisfies $g_1(x)=x_1x_m$.

\medskip\noindent
By the chain rule,
\[
\left.\frac{\partial}{\partial z_2}\right|_{z_1}\,x_1(z_1,z_2)
=\sum_{j=1}^m \frac{dt_j}{dz_2}\,
\frac{\partial}{\partial t_j}\Bigl(\sum_{k=1}^m\frac{1}{z_1-t_k}\Bigr)
=\sum_{j=1}^m\frac{1}{(z_1-t_j)^2}\,\frac{dt_j}{dz_2}.
\]
Hence, the system has a double pole at $z_1=t_j$ with coefficient $dt_j/dz_2$.

\medskip\noindent
Assume the poles are distinct (i.e.\ $(t_1,\dots,t_m)\notin\Delta$), so $t_j\neq t_k$
for $j\neq k$.  Fix $j$ and and expand near $z_1=t_j$ gives
\[
x_1(z_1,z_2)=\frac{1}{z_1-t_j}+R_j(z_1),
\qquad
R_j(z_1):=\sum_{k\ne j}\frac{1}{z_1-t_k},
\]
and
\[
x_m(z_1,z_2)=\frac{1}{z_1-t_j}\,H_j(z_1),
\qquad
H_j(z_1):=\frac{1}{\prod_{k\ne j}(z_1-t_k)}.
\]
Since $t_k\neq t_j$ for $k\ne j$, both $R_j$ and $H_j$ are holomorphic near $z_1=t_j$.
Therefore, as $z_1\to t_j$,
\[
R_j(z_1)=R_j(t_j)+O(z_1-t_j),\qquad
H_j(z_1)=H_j(t_j)+O(z_1-t_j).
\]
where
\[
H_j(t_j)=\frac{1}{\prod_{k\ne j}(t_j-t_k)}.
\]
Multiplying,
\begin{align*}
x_1x_m
&=\left(\frac{1}{z_1-t_j}+R_j(z_1)\right)\left(\frac{1}{z_1-t_j}H_j(z_1)\right)\\
&=\frac{H_j(z_1)}{(z_1-t_j)^2}+\frac{R_j(z_1)H_j(z_1)}{z_1-t_j}.
\end{align*}
Since $R_j$ and $h_j$ are holomorphic at $z_1=t_j$, the product $R_jH_j$ is holomorphic there, so
\[
\frac{R_j(z_1)H_j(z_1)}{z_1-t_j}
\]
has at most a simple pole and does not affect the $(z_1-t_j)^{-2}$ coefficient.
Thus the coefficient of $(z_1-t_j)^{-2}$ in $x_1x_m$ is $H_j(t_j)$.

\medskip\noindent
Using $G(x_1)=x_1x_m$ and comparing the coefficients of $(z_1-t_j)^{-2}$ gives
\[
\frac{dt_j}{dz_2}=H_j(t_j)=\frac{1}{\prod_{k\ne j}(t_j-t_k)},
\qquad j=1,\dots,m.
\]
Finally, the right-hand side blows up exactly when $t_j=t_k$ for some $j\ne k$,
i.e.\ precisely on the discriminant
\[
\Delta=\Bigl\{(t_1,\dots,t_m):\ \prod_{i<j}(t_i-t_j)^2=0\Bigr\}.
\]
\end{proof}

% With $b_1=2z_0$ and $b_2=z_0^2+2C$, we recover the same general solution as \eqref{general1} after transforming back to $(x_1,x_2)$.

%%%%%%%%%%%%%%%%%%%%%%%%%%%%%%%%%%%%%%%%%%%%%%%%%%%%%%%%%%%%%%%%%%%%%%%%%%%%%%%%%%%%%%%%%%%%%%%%%%
%%%%%%%%%%%%%%%%%%%%%%%%%%%%%%%%%%%%%%%%%%%%%%%%%%%%%%%%%%%%%%%%%%%%%%%%%%%%%%%%%%%%%%%%%%%%%%%%%%
%%%%%%%%%%%%%%%%%%%%%%%%%%%%%%%%%%%%%%%%%%%%%%%%%%%%%%%%%%%%%%%%%%%%%%%%%%%%%%%%%%%%%%%%%%%%%%%%%%
%%%%%%%%%%%%%%%%%%%%%%%%%%%%%%%%%%%%%%%%%%%%%%%%%%%%%%%%%%%%%%%%%%%%%%%%%%%%%%%%%%%%%%%%%%%%%%%%%%
\section{Annular Laurent series expansions}
In the previous section we \emph{varied} $z_2$ and studied how the commuting
$z_2$--flow moves the pole parameters $t_j(z_2)$; pole collision ($t_i=t_j$)
produces degeneration.
In what follows we take the complementary viewpoint: we \emph{freeze} $z_2$
(so $t_j=t_j(z_2)$ are fixed complex numbers) and analyze the local
Laurent expansion of $x_1$ as a function of $z_1$.

The key point is that, when several poles are present, the expansion near a
chosen pole depends on the annular region in which $z_1$ lies relative to the
other poles. In particular, we can show that negative Kovalevskaya exponents correspond to annulus indices, which count the number of poles inside the region of expansion. This explains how a single movable pole can support multiple distinct Laurent branches.

We start by recalling the general solution in partial–fraction
form \eqref{eq:x1-xm-poles},
\[
x_1(t)=\sum_{j=1}^{m}\frac{1}{z_1-t_j},
\]
where $t_1,\dots,t_m$ are the poles of the solution. We assume $t_j\neq t_k$ for $j\neq k$ Fixing one pole $t_1$ and
introducing the shifted variable
\[
w:=z_1-t_1,\qquad \delta_j:=t_j-t_1\quad (j\ge2),
\]
we may write
\begin{equation}\label{eq:cx1-shift}
x_1(w)=\frac{1}{w}+\sum_{j=2}^{m}\frac{1}{w-\delta_j}.
\end{equation}

Thus the local behavior of $x_1$ near $w=0$ is governed by how the remaining poles
$\delta_j$ are distributed relative to the domain in which $w$ varies.

\subsection{Geometric series expansions}

Each term $(w-\delta_j)^{-1}$ admits two different convergent expansions,
depending on whether $|w|$ is smaller or larger than $|\delta_j|$:
\begin{itemize}
\item \textbf{Inside the disc} ($|w|<|\delta_j|$):
\begin{equation}\label{eq:inner-expansion}
\frac{1}{w-\delta_j}
=-\frac{1}{\delta_j}\frac{1}{1-w/\delta_j}
=-\sum_{k=0}^{\infty}\frac{w^k}{\delta_j^{k+1}},
\end{equation}
which contains no $w^{-1}$ term.
\item \textbf{Outside the disc} ($|w|>|\delta_j|$):
\begin{equation}\label{eq:outer-expansion}
\frac{1}{w-\delta_j}
=\frac{1}{w}\frac{1}{1-\delta_j/w}
=\frac{1}{w}\sum_{k=0}^{\infty}\Big(\frac{\delta_j}{w}\Big)^k,
\end{equation}
whose leading term is $1/w$.
\end{itemize}

Which expansion is valid therefore depends on the region of the complex plane in
which $w$ is allowed to lie.

\subsection{Annular decomposition}

Let the moduli of the shifted poles be ordered as
$0=|\delta_1|<|\delta_2|<\cdots<|\delta_m|$, and define the annuli
\begin{equation}\label{eq:annuli}
A_n=\{\,w\in\C:\ |\delta_n|<|w|<|\delta_{n+1}|\,\},\qquad n=1,\dots,m,
\end{equation}
with the convention $|\delta_{m+1}|=\infty$.  Thus $A_1$ is a punctured disc,
$A_m$ is the exterior region, and the intermediate $A_n$ are genuine annuli
(Figure~\ref{fig:annular}).

Fix any $w\in A_n$ and split the poles $\{\delta_2,\dots,\delta_m\}$ into
\begin{equation*}
\begin{aligned}
&I_{\rm in}(n):=\{\,j\in\{2,\dots,m\}:\ |\delta_j|<|w|\,\},\\
&I_{\rm out}(n):=\{\,j\in\{2,\dots,m\}:\ |\delta_j|>|w|\,\}.
\end{aligned}
\end{equation*}
so that $|I_{\rm in}(n)|=n-1$ and $|I_{\rm out}(n)|=m-n$.

\begin{figure}[h]
\centering
\begin{tikzpicture}[scale=1.0, >=Latex]

\def\rTwo{1.25}
\def\rThree{2.25}

\draw[thick] (0,0) circle (\rTwo);
\draw[thick] (0,0) circle (\rThree);

\fill (0,0) circle (1.6pt);
\node[below left] at (0,0) {$0$};

\fill (\rTwo,0) circle (1.4pt);
\node[below] at (\rTwo,0) {$\delta_2$};

\fill ({\rThree*cos(30)},{\rThree*sin(30)}) circle (1.4pt);
\node[above right] at ({\rThree*cos(30)},{\rThree*sin(30)}) {$\delta_3$};

\fill ({3.0*cos(210)},{3.0*sin(210)}) circle (1.4pt);
\node[below left] at ({3.0*cos(210)},{3.0*sin(210)}) {$\delta_4$};

\node at (0,0.55) {$A_1$};
\node at (0,1.75) {$A_2$};

\fill (1.6,0.9) circle (1.4pt);
\node[above] at (1.6,0.9) {$w$};

\end{tikzpicture}
\caption{Annular decomposition of the complex plane centered at a fixed pole \(w=0\).
The annuli are determined by the ordered moduli \(|\delta_j|\) of the remaining poles.}
\label{fig:annular}
\end{figure}

\begin{proposition}\label{annulus-index}
On each annulus $A_n$, the function $x_1(w)$ admits a Laurent expansion
\begin{equation}\label{eq:annular-y-form}
x_1(w)=\frac{n}{w}+H_n(w),
\end{equation}
where $H_n(w)$ is holomorphic on $A_n$ and has the explicit expansion
\begin{equation}\label{eq:Hn-explicit}
H_n(w)=
\sum_{j\in I_{\rm in}(n)}\sum_{k=1}^{\infty}\frac{\delta_j^{k}}{w^{k+1}}
\;-\;
\sum_{j\in I_{\rm out}(n)}\sum_{k=0}^{\infty}\frac{w^{k}}{\delta_j^{k+1}}.
\end{equation}
In particular, the coefficient of $w^{-1}$ is exactly $n =1+|I_{\rm in}(n)|$.
\end{proposition}

\begin{proof}
For $j\in I_{\rm in}(n)$ we use the outer expansion \eqref{eq:outer-expansion}, while
for $j\in I_{\rm out}(n)$ we use the inner expansion \eqref{eq:inner-expansion}.
Substituting into \eqref{eq:cx1-shift} gives
\[
x_1(w)=\frac1w+\sum_{j\in I_{\rm in}(n)}\!\left(\frac1w+\sum_{k\ge1}\frac{\delta_j^k}{w^{k+1}}\right)
-\sum_{j\in I_{\rm out}(n)}\sum_{k\ge0}\frac{w^k}{\delta_j^{k+1}}.
\]
The coefficient of $w^{-1}$ is therefore
$1+|I_{\rm in}(n)|=1+(n-1)=n$, and the remaining terms give $H_n(w)$.
\end{proof}

\begin{remark}
Proposition~\ref{annulus-index} gives an analytic interpretation with a geometric origin to the negative
Kovalevskaya exponent: it is the \emph{annulus index} $n$, counting how many poles
of $x_1$ lie inside the region of expansion.  Thus a single movable pole admits
$m$ distinct Laurent expansions, corresponding to the $m$ regions $A_n$ (inner
disc, intermediate annuli, and exterior), each producing a different negative
resonance.
\end{remark}

%%%%%%%%%%%%%%%%%%%%%%%%%%%%%%%%%%%%%%%%%%%%%%%%%%%%%%%%%%%%%%%%%%%%%%%%%%%%%%%%%%%%%%%%%%%%%%%%%%

\noindent\textbf{Acknowledgments.}  
The author wishes to express sincere gratitude to Professor Hayato Chiba for his valuable guidance and instruction.

%%%%%%%%%%%%%%%%%%%%%%%%%%%%%%%%%%%%%%%%%%%%%%%%%%%%%%%%%%%%%%%%%%%%%%%%%%%%%%%%%%%%%%%%%%%%%%%%%%
%%%%%%%%%%%%%%%%%%%%%%%%%%%%%%%%%%%%%%%%%%%%%%%%%%%%%%%%%%%%%%%%%%%%%%%%%%%%%%%%%%%%%%%%%%%%%%%%%%
%%%%%%%%%%%%%%%%%%%%%%%%%%%%%%%%%%%%%%%%%%%%%%%%%%%%%%%%%%%%%%%%%%%%%%%%%%%%%%%%%%%%%%%%%%%%%%%%%%
%%%%%%%%%%%%%%%%%%%%%%%%%%%%%%%%%%%%%%%%%%%%%%%%%%%%%%%%%%%%%%%%%%%%%%%%%%%%%%%%%%%%%%%%%%%%%%%%%%


\begin{thebibliography}{99}
\setlength{\baselineskip}{0pt}
\bibitem{Adl}
Adler, M., Van Moerbeke, P.and  Vanhaecke, P. (2013). Algebraic integrability, Painlevé geometry and Lie algebras (Vol. 47). Springer Science and Business Media.


\bibitem{Bureau}
Bureau, F. J. (1964). Differential equations with fixed critical points. Annali di Matematica pura ed applicata, 64(1), 229-364.


\bibitem{Chi}
Chiba, H. (2015). A compactified Riccati equation of Airy type on a weighted projective space. RIMS Kokyuroku Bessatsu,(to appear).


\bibitem{Chi1}
Chiba, H. (2015). Kovalevskaya exponents and the space of initial conditions of a quasi-homogeneous vector field. Journal of Differential Equations, 259(12), 7681-7716.

\bibitem{Chi2}
Chiba, H. (2016). The first, second and fourth Painlevé equations on weighted projective spaces. Journal of Differential Equations, 260(2), 1263-1313.

\bibitem{Chi3}
Chiba, H. (2016). The third, fifth and sixth Painlevé equations on weighted projective spaces. SIGMA. Symmetry, Integrability and Geometry: Methods and Applications, 12, 019.


\bibitem{Chi4}
Chiba, H. (2024). Weights, Kovalevskaya exponents and the Painlevé property. Annales de l’Institut Fourier, 74(2), 811–848.


\bibitem{Cosgrove_Chazy}
Cosgrove, C. M. (2000). Chazy classes IX–XI of third‐order differential equations. Studies in applied mathematics, 104(3), 171-228.

\bibitem{Cosgrove_P2}
Cosgrove, C. M. (2000). Higher‐order Painlevé equations in the polynomial class I. Bureau symbol P2. Studies in applied mathematics, 104(1), 1-65.

\bibitem{deLucas2017}
Grundland, A. M. and de Lucas, J. (2017). A Lie systems approach to the Riccati hierarchy and partial differential equations. Journal of Differential Equations, 263(1), 299-337.

\bibitem{Goriely}
Goriely, A. (2001). Integrability and nonintegrability of dynamical systems (Vol. 19). World Scientific.

\bibitem{Guha2006}
Guha, P. (2006). Riccati chain, higher order Painlevé type equations and stabilizer set of Virasoro orbit (MPI MIS Preprint No. 143). Max Planck Institute for Mathematics in the Sciences.

\bibitem{Kovalevski}
Kowalevski, S. (1889). Sur le problème de la rotation d'un corps solide autour d'un point fixe. Acta Math. 12. 177-232.

\bibitem{HuYan}
Hu, J. and Yan, M. (2013). Painlevé test and the resolution of singularities for integrable equations. arXiv preprint arXiv:1304.7982.

\bibitem{Okamoto}
Okamoto, K. (1979). Sur les feuilletages associés aux équation du second ordre à points critiques fixes de P. Painlevé Espaces des conditions initiales. Japanese journal of mathematics. New series, 5(1), 1-79.

\bibitem{Painleve}
Painleve, P. (1902). Sur les équations différentielles du second ordre et d'ordre supérieur dont l'intégrale générale est uniforme. Acta mathematica, 25(1), 1-85.

\bibitem{Sakai}
Sakai, H. (2001). Rational surfaces associated with affine root systems and geometry of the Painlevé equations. Communications in Mathematical Physics, 220(1), 165–229.

\bibitem{Shimomura}
Shimomura, S. (2001). Pole loci of solutions of a degenerateGarnier system. Nonlinearity, 14(2), 193.

\bibitem{Yoshida}
Yoshida, H. (1983). Necessary condition for the existence of algebraic first integrals: I: Kowalevski's exponents. Celestial mechanics, 31(4), 363-379.

\bibitem{ZhouChiba2025}
Zhou, C. and Chiba, H. (2025). Quasi-Homogeneous Integrable Systems: Free Parameters, Kovalevskaya Exponents, and the Painlev\'e Property. arXiv preprint arXiv:2505.24330.

\end{thebibliography}
\end{document}